\documentclass[11pt,leqno]{amsart}

\usepackage{amsmath}
\usepackage{amssymb}
\usepackage{a4wide}
\usepackage{graphics}
\usepackage{stmaryrd}
\usepackage{enumerate}
\newcommand{\Subsection}[1]{\subsection{ #1} ${}^{}$}
\newcommand{\Subsubsection}[1]{\subsubsection{ #1} ${}^{}$}

\newtheorem{theorem}{Theorem}[section]
\newtheorem{lemma}[theorem]{Lemma}

\newtheorem{remark}[theorem]{Remark} 
\newtheorem{corollary}[theorem]{Corollary}
\newtheorem{example}[theorem]{Example}

\newcounter{hypo}
\newcounter{hyphyp}

\def\N{{\mathbb N}} 
\def\R{{\mathbb R}} 

\def\Z{{\mathbb Z}}

\def\T{{\mathbb T}}

\def\CO{\mathcal {O}}

\def\im{\mathop{\rm Im}\nolimits}

\def\supp{\mathop {\rm supp}\nolimits}

\def\<{\langle}
\def\>{\rangle}

\makeatletter
 \@addtoreset{equation}{section}
 \makeatother

\newcommand{\dN}{\mathbb{N}}
\newcommand{\dT}{\mathbb{T}}
\newcommand{\dR}{\mathbb{R}}
\newcommand{\dE}{\mathbb{E}}
\newcommand{\dP}{\mathbb{P}}

\newcommand{\cN}{\mathcal{N}}

\newcommand{\cW}{\mathcal{W}}

\newcommand{\wh}{\widehat}

\newcommand{\ind}{\mbox{1}\kern-.25em \mbox{I}}

\def\build#1_#2^#3{\mathrel{\mathop{\kern 0pt#1}\limits_{#2}^{#3}}}

\def\videbox{\mathbin{\vbox{\hrule\hbox{\vrule height1ex \kern.5em
\vrule height1ex}\hrule}}}

\title[Spectrum of the product of Toeplitz matrices]{Spectrum of the product of Toeplitz matrices with application in probability}

\author[B. Bercu]{Bernard Bercu}
\address{Institut de Math\'ematiques de Bordeaux, 
UMR 5251, Universit\'e Bordeaux 1, 351 cours de la lib\'eration, 
33405 Talence cedex, France.}
\email{Bernard.Bercu@math.u-bordeaux1.fr}

\author[J.-F. Bony]{Jean-Fran\c{c}ois Bony}
\email{Jean-Francois.Bony@math.u-bordeaux1.fr}

\author[V. Bruneau]{Vincent Bruneau}
\email{Vincent.Bruneau@math.u-bordeaux1.fr}



\begin{document}

\begin{abstract}
We study the spectrum of the product of two Toeplitz operators. Assume that the symbols of these operators are continuous and real-valued and that one of them is non-negative. We prove that the spectrum of the product of finite section Toeplitz matrices converges to the spectrum of the product of the semi-infinite Toeplitz operators. We give an example showing that the supremum of this set is not always the supremum of the product of the two symbols. Finally, we provide an application in probability which is the first motivation of this study. More precisely, we obtain a large deviation principle for Gaussian quadratic forms.

\end{abstract}

\maketitle


\section{Introduction}

For any bounded measurable real function $f$ on the torus
$\T = \R / 2 \pi \Z$, the $\ell^2(\dN)$ Toeplitz and Hankel operators are respectively given by
\begin{equation}
\label{deftoephank}
T(f)=\Big(\wh{f}_{i-j}\Big)_{\!i,j\geq 0}
\qquad \text{ and } \qquad
H(f)=\Big(\wh{f}_{i+j+1}\Big)_{\!i,j\geq 0}
\end{equation}
where $(\wh{f}_n)$ stands for the sequence of Fourier coefficients of $f$. We refer to the book of B\"{o}ttcher and Silbermann \cite{BoSi06_01} for a general presentation of Toeplitz operators.
A well-known identity between the product $T(f)T(g)$ and $T(fg)$ is
\begin{equation}
\label{prodtoep}
T(fg)-T(f)T(g)=H(f)H(J g)
\end{equation}
where $J g(x)=g(-x)$. In all the sequel, we shall denote by $T_n(f)$ the section of
order $n\geq 0$ of $T(f)$ which means that $T_n(f)$ is identified with $\Pi_n T(f) \Pi_n$
where $\Pi_n$ is the orthogonal projection operator on the Fourier modes $0 , \ldots ,n$. The analogue of identity
(\ref{prodtoep}) for finite section Toeplitz matrices is given by the formula of Widom \cite{Wi76_01}
\begin{equation}
\label{prodtoepn}
T_n(fg)-T_n(f)T_n(g)= \Pi_{n} H(f) H(J g) \Pi_{n} + \widetilde{\Pi}_{n} H(J f)H(g)\widetilde{\Pi}_{n} ,
\end{equation}
where $\widetilde{\Pi}_{n} = \Pi_{n} e^{i n x} J \Pi_{n}$.

The asymptotic behavior of the spectrum of a single Toeplitz matrix, say $T_n(f)$, dates back to Szeg\"o (see the book of Grenander and Szeg\"o \cite{GrSz58_01}).
However, it is much more difficult to deal with the spectrum of the product of two Toeplitz matrices
$T_n(f)T_n(g)$. Several authors have investigated the asymptotic behavior of the spectrum of 
$T_n(f)T_n(g)$. More precisely, assume that $f , g \in L^{\infty} ( \T , \R )$ with $ g \geq 0$ and denote by $\lambda_{0}^{n}, \ldots, \lambda_n^n$ the eigenvalues of the product $T_n(f)T_n(g)$. It was shown in \cite{BeGaRo97_01} and by Serra Capizzano \cite{Se01_01} that

\begin{equation}
\frac{1}{n}\sum_{k=0}^{n} \varphi ( \lambda_k^n ) \longrightarrow \frac{1}{2 \pi} \int_{\T} \varphi (f g) (x) \, d x ,
\end{equation}
for any $\varphi \in C^{0} (\R )$. There also exits a wide literature on the asymptotic behavior of the determinant of Toeplitz matrices (see Gray \cite{Gr06_01} for a review). However, to the best of our knowledge, no result is available on the asymptotic behavior of the maximum as well as on the minimum eigenvalues of the product $T_n(f)T_n(g)$.

The purpose of this paper is to prove that the eigenvalues of $T_n(f)T_n(g)$ converge to the spectrum of the limit operator $T(f)T(g)$. In particular, the maximum and the minimum eigenvalues of
$T_n(f)T_n(g)$ both converge to the maximum and minimum of the spectrum of $T(f)T(g)$. It allows us to deduce a direct application in probability on large deviations for quadratic forms of Gaussian processes.

Our approach is based on semi-classical analysis and scattering theory. In order to prove the convergence of the spectrum, we construct quasimodes which are approximative eigenvectors. More precisely, we show that eigenvectors associated with the eigenvalues outside the essential spectrum are localized in the lower and higher frequencies.

The paper is organized as follows. In Section \ref{a49}, we state the main results. Section \ref{a50} is devoted to the application in probability. Our functional point of view on Toeplitz operators and the tools used in the paper are given in Section \ref{a53}. The convergence of the spectrum is proved in Section \ref{a51}, while the rest of the proofs is postponed in Section \ref{a52}. Finally, in Section \ref{a41}, we propose an alternative proof of Douglas's theorem.

\section{Main results on Toeplitz operators}
\label{a49}

According to the identification of $L^2(\dT)$ with $l^2(\Z)$ (which contains $l^2(\dN)$) by the Fourier transform, we can view all the operators as operators on $L^2(\dT)$. In this case, $0$ is necessarily an eigenvalue of infinite multiplicity of each operators $T(f)$, $T_n(f)$ (see Section \ref{a32} for the precise definition of the Toeplitz operators). In this paper, our operators will be considered as operators on $\im \Pi$ and $\im \Pi_{n}$. We shall denote by $\sigma ( A)$ the spectrum of $A$. We have

\begin{lemma}\sl \label{a30}
Let $f ,g \in L^{\infty} (\T )$ with $g$ real-valued and non-negative. Then, on $\im \Pi$, we have
\begin{equation*}
\sigma \big( T (g)^{1/2} T (f) T (g)^{1/2} \big) = \sigma ( T (f) T (g) ) .
\end{equation*}
The same result holds for $T_{n}$ on $\im \Pi_{n}$ and the multiplicity of the eigenvalues is the same.
\end{lemma}

\begin{remark}\sl \label{a31}
This property is also true for $T$ and $T_{n}$ on $L^{2} (\T )$.
\end{remark}

Please note that $\sigma (A B) = \sigma (B A)$ is not true in general. By example, for $f(x) = e^{ix}$ and $g(x)= e^{-ix}$, we have $\sigma (T(f)T(g))=\{ 0,1 \}$ and $\sigma (T(g)T(f))=\{ 1 \}$ on $\im \Pi$.

Assume now that $f ,g \in L^{\infty} ( \T )$ are both real-valued with $g \geq 0$. Then, $T (g)^{1/2} T (f) T (g)^{1/2}$ and $T_{n} (g)^{1/2} T_{n} (f) T_{n} (g)^{1/2}$ are self-adjoint and their spectrum is real. Let us introduce the maximum and minimum eigenvalues of the product of the Toeplitz operators of finite rank on $\im \Pi_{n}$
\begin{align*}
\lambda_{\max}^{n} (f ,g) =& \max \sigma ( T_{n} (f) T_{n} (g) )  \\
\lambda_{\min}^{n} (f ,g) =& \min \sigma ( T_{n} (f) T_{n} (g) ) ,
\end{align*}
and the extrema of the spectrum of the product of the Toeplitz operators on $\im \Pi$
\begin{align*}
\lambda_{\max} (f ,g) =& \max \sigma ( T (f) T (g) )   \\
\lambda_{\min} (f ,g) =& \min \sigma ( T (f) T (g) ) .
\end{align*}
One can observe that, in general, we don't know if $\lambda_{\max} (f ,g)$ and $\lambda_{\min} (f ,g)$ are eigenvalues.

\begin{theorem}\sl \label{a29}
Assume that $f , g \in C^{0} (\T, \R )$ with $g \geq 0$. Then, as $n$ goes to infinity,
\begin{equation*}
\lambda_{\max}^{n} (f ,g) \longrightarrow \lambda_{\max} (f ,g) \quad \text{ and } \quad \lambda_{\min}^{n} (f ,g) \longrightarrow \lambda_{\min} (f ,g) .
\end{equation*}
\end{theorem}

In fact, we will prove the following result which implies Theorem \ref{a29}.

\begin{theorem}\sl \label{a8}
Assume that $f , g \in C^{0} (\T,\R )$ with $g \geq 0$. For $\lambda \in \R$, the following properties are equivalent:

{\it i)} $\lambda \in \sigma ( T (f) T (g) )$,

{\it ii)} there exists $\lambda_{n} \in \sigma ( T_{n} (f) T_{n} (g) )$ such that $\lambda_{n} \to \lambda$,

{\it iii)} there exists a subsequence $N$ of $\N$ and $\lambda_{N} \in \sigma ( T_{N} (f) T_{N} (g) )$ such that $\lambda_{N} \to \lambda$.
\end{theorem}

\begin{remark}\sl
In the previous theorems, the operators $T_{n}$ and $T$ are considered on $\im \Pi_{n}$ and on $\im \Pi$ respectively. But, our results are also true if these operators are considered on $L^{2} (\T )$ (we only add the eigenvalue $0$).
\end{remark}

As shown in the next example, it is not true in general that $\lambda_{\max} (f,g) = \sup (f g)$ or $\lambda_{\min} (f,g) = \inf (f g)$. Note also that the norm of $T (g)^{1/2} T (f) T(g)^{1/2}$ is not always equal to $\Vert f g \Vert_{\infty}$ or $\Vert f \Vert_{\infty} \Vert g \Vert_{\infty}$. The situation is totally different from the case of a single Toeplitz operator $T (f)$ with real continuous symbol as $\lambda_{\max} (f,1) = \sup (f)$ and $\lambda_{\min} (f ,1) = \inf (f)$.

\begin{example}\sl \label{a40}
Let $a, \theta \in \dR$ with $|\theta|<1$ and consider $f,g \in C^{0} (\T)$
given by 
\begin{equation*}
f(x)=a+\cos(x) \hspace{0.5cm}\text{and}\hspace{0.5cm}g(x)=\frac{1}{1+\theta^2-2\theta \cos(x)}.
\end{equation*}
We have $\Vert f \Vert_{\infty} = \vert a \vert +1$ and $\Vert g \Vert_{\infty} = ( 1 - \vert \theta \vert )^{-2}$.
The function $g$ is simply the spectral density of a Gaussian autoregressive process \cite{BeGaRo97_01}.
If $\theta =0$, $g=1$ and the product $T(f)T(g)$ reduces to $T(f)$.
Consequently, $\lambda_{\max} (f ,g)=a+1$ and  $\lambda_{\min} (f ,g)=a-1$.
If $\theta \neq 0$, denote
\begin{equation*}
a_{\theta}=-\frac{(1+\theta)}{2\theta}
\hspace{1cm}\text{and}\hspace{1cm}
b_{\theta}=-\frac{(1-\theta)}{2\theta}.
\end{equation*}
We shall show in section \ref{a38} that $\lambda_{\min} (f ,g)< \inf(fg)$
and $\lambda_{\max} (f ,g)> \sup(fg)$. More precisely, if $\theta >0$ then
$\lambda_{\max} (f ,g)=\sup(fg)$ while
\begin{equation*}
\lambda_{\min} (f ,g)=\frac{1}{-4\theta(1+a\theta)}<\inf ( f g) = \min \Big( \frac{a-1}{(1+ \theta )^{2}} , \frac{a+1}{(1- \theta )^{2}} \Big)
\end{equation*}
if $a\in ]a_{\theta},b_{\theta}[$ 
and $\lambda_{\min} (f ,g)= \inf(fg)$ otherwise. Moreover, if $\theta <0$ then
$\lambda_{\min} (f ,g)=\inf(fg)$ while
\begin{equation*}
\lambda_{\max} (f ,g)=\frac{1}{-4\theta(1+a\theta)} > \sup (f g) = \max \Big( \frac{a-1}{(1+ \theta )^{2}} , \frac{a+1}{(1- \theta )^{2}} \Big)
\end{equation*}
if $a \in ]a_{\theta},b_{\theta}[$ 
and $\lambda_{\max} (f ,g)= \sup(fg)$ otherwise. 
\end{example}

\section{Application in probability}
\label{a50}

Let $(X_{n})$ be a centered stationary real Gaussian process with bounded positive spectral density 
$g$ which means that
\begin{equation*}
\dE[X_{j}X_{k}]=\frac{1}{2\pi}\int_{\T } \exp(i(j-k)x)g(x) \, d x.
\end{equation*}
We assume in all the sequel that $g$ is not the zero function. For any real continuous function $f \in C^{0}(\T )$, we are interested in the asymptotic behavior of
\begin{equation}
\label{defperiodo}
\cW_n(f)=\frac{1}{2\pi n}\int_{\dT} f(x) \bigg\vert \sum_{j=0}^{n}
X_{j}\exp(ijx)\bigg\vert^{2} dx.
\end{equation}
The purpose of this section is to provide the last step in the analysis of the large deviation
properties of $(\cW_n(f))$ by establishing a large deviation principle (LDP) for 
$(\cW_n(f))$ in the spirit of the original work of \cite{BeGaRo97_01} or of Bryc and Dembo \cite{BrDe97_01}. We refer the reader to the book of Dembo and Zeitouni \cite{DeZe98_01} for the general theory on large deviations. The covariance matrix associated with the vector $X^{(n)}=(X_{0},\ldots,X_{n})^t$
is $T_{n} (g)$. Consequently, it immediately follows from
(\ref{defperiodo}) that
\begin{equation}
\label{decomperio}
\cW_n(f) =\frac{1}{n} X^{(n)t}T_{n} (f)X^{(n)}=\frac{1}{n} Y^{(n)t}T_{n} (g)^{1/2} T_{n} (f) T_{n} (g)^{1/2}Y^{(n)}
\end{equation}
where the vector $Y^{(n)}$ has a Gaussian $\cN(0,I_n)$ distribution. In order to investigate the large deviation 
properties of $(\cW_n(f))$, it is necessary to calculate the normalized cumulant generating function
given, for all $t\in \dR$, by
$$ L_{n}(t) = \frac{1}{n}\log \dE \big[ \exp(n t \cW_n(f)) \big].$$
For convenience and in all the sequel, we use of the notation
that $\log t=-\infty$ if $t\leq 0$. We deduce from (\ref{decomperio}) and 
standard Gaussian calculation that
for all $t\in \dR$
\begin{align*}
L_n(t) =& -\frac{1}{2n}\log \det \big( I_n-2t T_{n} (g)^{1/2} T_{n} (f)T_{n} (g)^{1/2} \big) ,  \\
=& -\frac{1}{2n}\sum_{k=0}^n \log(1 -2 t\lambda_k^{n})
\end{align*}
where $\lambda_0^n, \ldots, \lambda_n^n$  
are the eigenvalues of $T_{n} (g)^{1/2} T_{n} (f) T_{n} (g)^{1/2}$ which are also
the eigenvalues of $T_{n} (f) T_{n} (g)$ from Lemma \ref{a30}. For all $t\in \dR$, let
\begin{equation*}
L_{f g}(t)=-\frac{1}{4\pi}\int_{\T} \log (1-2tf(x)g(x)) \, d x ,
\end{equation*}
and denote by $I_{fg}$ its Fenchel-Legendre transform
\begin{equation*}
I_{fg}(x)= \sup_{t\in \dR}\{xt-L_{fg}(t)\}.
\end{equation*}
Furthermore, for all $x\in \dR$, let
\begin{equation}
\label{defJfg}
J_{f g}(x)= \left \{ 
\begin{aligned}
&I_{f g}(a)+\frac{1}{2{\lambda_{\min}(f,g)}}(x-a) \quad &&\text{if } x \in ]-\infty,a] \\
&I_{f g}(x) &&\text{if } x \in ]a,b[   \\
&I_{f g}(b)+\frac{1}{2{\lambda_{\max}(f,g)}}(x-b) &&\text{if } x \in [b,+\infty[ 
\end{aligned}  \right.
\end{equation}
where $a$ and $b$ are the extended real numbers given by
\begin{equation*}
a= L_{fg}^{\prime}\left( \frac{1}{2{\lambda_{\min}(f,g)}}\right)
\end{equation*}
if $\lambda_{\min}(f,g)<0$ and $\lambda_{\min}(f,g)<\inf(fg)$,
$a=-\infty$ otherwise, while
\begin{equation*}
b= L_{fg}^{\prime}\left( \frac{1}{2{\lambda_{\max}(f,g)}}\right)
\end{equation*}
if $\lambda_{\max}(f,g)>0$ and $\lambda_{\max}(f,g)>\sup(f g)$,
$b=+\infty$ otherwise. We immediately deduce from Theorem 1 of \cite{BeGaRo97_01} together with Theorem \ref{a29}, that the LDP holds for $(\cW_n(f))$.

\begin{corollary}\sl \label{corproba}
The sequence $(\cW_n(f))$ satisfies an LDP with good rate function
$J_{fg}$. More precisely, for any closed set $F  \subset \dR$
$$
\limsup_{n\rightarrow\infty}\frac{1}{n}\log \dP(\cW_n(f)\in F)\leq
-\inf_{x\in F}J_{fg}(x),
$$
while for any open set $G \subset \dR$
$$
\liminf_{n\rightarrow\infty}\frac{1}{n}\log \dP(\cW_n(f)\in G)\geq 
-\inf_{x\in G}J_{fg}(x).
$$
\end{corollary}

\begin{remark}\sl
Denote by $\mu$ the derivative of $L_{fg}$ at point zero
$$ \mu=\frac{1}{2\pi}\int_{\dT} f(x)g(x)dx. $$
Then, we have $J_{fg}(\mu)=0$ and it follows from Corollary \ref{corproba}
that for all $x>\mu$
\begin{equation*}
\lim_{n\rightarrow\infty}\frac{1}{n}\log \dP(\cW_n(f)\geq x)=-J_{fg}(x) ,
\end{equation*}
whereas for all $x<\mu$
\begin{equation*}
\lim_{n\rightarrow\infty}\frac{1}{n}\log \dP(\cW_n(f)\leq x)=-J_{f g} (x).
\end{equation*}
\end{remark}

\section{Toeplitz operators and functional calculus}
\label{a53}

In this section, we interpret the projection operators $\Pi_{n}$ and $\Pi$ as spectral projectors of the derivation operator. We also introduce the main ingredients of the proofs.

\Subsection{A functional point of view}
\label{a32}

Our approach consists to view the Toeplitz operators $T(f)$ and $T_{n} (f)$ as the cut-off, in frequencies, of the operator of multiplication by $f$.

To be more precise, let us introduce the Fourier transform, $\mathcal{F}: L^2(\dT) \to l^2(\Z)$, defined by
\begin{equation*}
(\mathcal{F}u)_k = \widehat{u}_k = \frac{1}{2 \pi} \int_{- \pi}^{\pi} u (x) e^{-i k x} d x .
\end{equation*}
The operator $\mathcal{F}$ is an isomorphism. We denote by $\mathcal{F}^{-1}$ its inverse, and we introduce the projections $\wh{\Pi}$ and $\wh{\Pi}_n$ as
\begin{align*}
\widehat{\Pi}:& \widehat{u} \in l^2(\Z) \longmapsto (\ldots, 0, 0, \widehat{u}_0, \widehat{u}_1, \ldots ) \in l^2(\Z) \\
\widehat{\Pi}_n:& \widehat{u} \in l^2(\Z) \longmapsto (\ldots, 0, 0, \widehat{u}_0, \widehat{u}_1, \ldots , \widehat{u}_n,0, 0,  \ldots) \in l^2(\Z ) .
\end{align*}

On the other hand, if we identify $f \in L^{\infty}(\T )$ with $L(f)$, the bounded operator defined on $L^{2} ( \T )$ by
\begin{equation*}
u \in L^2(\T ) \longmapsto f u \in L^2(\T ) ,
\end{equation*}
we have
\begin{equation*}
T(f)=\Pi \, f \,\Pi  \quad \text{ and } \quad T_n(f)= \Pi_n \, f \,\Pi_n,
\end{equation*}
with $\Pi=\mathcal{F}^{-1}\wh{\Pi} \mathcal{F} $ and $\Pi_n=\mathcal{F}^{-1}\wh{\Pi}_n \mathcal{F}$. In the following, we will systematically identify $f$ with the operator $L (f)$.

On the other hand, since $\frac{1}{i} \frac{d}{dx}(e^{ikx})= k e^{ikx}$, the derivation operator 
$D$ defined on 
\begin{equation*}
 H^1(\dT) = \{ u \in L^2(\dT); \;  \frac{d}{dx}u \in  L^2(\dT)\} = \{ u \in L^2(\dT); \; (k \hat{u}_k)_k \in l^2(\Z)\}
\end{equation*}
by 
\begin{equation*}
 D: u \in H^1(\dT) \longmapsto \frac{1}{i} \frac{d}{dx} u \in L^2(\dT)
\end{equation*}
is self-adjoint on $L^2(\T )$ and $\mathcal{F} D \mathcal{F}^{-1}$ is the diagonal operator $(k\delta_{k,j})_{k,j \in \Z}$.

For any bounded Borel function $\varphi$, the bounded operator $\varphi (D)$ is defined with the help of the spectral theorem for self-adjoint operators (see Theorem VIII.5 of \cite{ReSi80_01}). It satisfies
\begin{equation*}
\varphi(D)=\mathcal{F}^{-1} \varphi(k) \mathcal{F} ,
\end{equation*}
where $\varphi(k)$ is identified with the operator $L ( \varphi )$
\begin{equation*}
\widehat{u} \in l^{2} (\Z ) \longmapsto (\ldots , \varphi(k) \widehat{u}_{k} , \ldots ) \in l^{2} (\Z ) .
\end{equation*}
In particular, if ${\bf 1}_{I}$ denotes the indicator function of the interval $I$, we have
\begin{equation*}
{\bf 1}_{[0, + \infty[}(D) = \Pi  \qquad \text{ and } \qquad {\bf 1}_{[0, n]}(D) ={\bf 1}_{[0, 1]}(n^{-1}\, D)= \Pi_n.
\end{equation*}
Moreover, note that if supp$(\varphi) \subset [a,b]$, and if we identify $e^{ikx}$ with the operator of multiplication by $e^{ikx}$, we have the trivial properties
\begin{equation*}
{\bf 1}_{[a,b]}(D) \, \varphi(D)=\varphi(D) \qquad \text{ and } \qquad {\bf 1}_{[a,b]}(D)\, e^{ikx} = e^{ikx}\, {\bf 1}_{[a-k,b-k]}(D) .
\end{equation*}

In the rest of the paper, a function is a $o_{a \to b}^{c} (1)$ if, for each $c$ fixed, the function goes to $0$ as $a$ goes to $b$. In the same way, a function is a $\CO^{c} (1)$ if, for each $c$ fixed, the function is a $\CO (1)$.

\Subsection{A commutator estimate}

In this subsection, we recall a standard result of the functional analysis. For $\rho \in\R$, we denote by $S^{\rho} ( \R )$ the class of functions $\varphi$ in $C^{\infty} (\R)$ such that
\begin{equation*}
\vert \partial_{s}^{k} \varphi (s) \vert \leq C_{k} \< s \>^{\rho -k},
\end{equation*}
for $k \geq 0$. Here $\< x \> = (1 + \vert  x \vert^{2})^{1/2}$.

\begin{lemma}[Lemma C.3.2 of \cite{DeGe97_01}]\sl \label{a25}
Let $A,B$ be self-adjoint operators on a Hilbert space with $B$ and $[A ,B]$ bounded. If $\varphi \in S^{\rho} (\R)$ with $\rho <1$, then
\begin{equation*}
\Vert [ \varphi (A) , B ] \Vert \leq C_{\varphi} \Vert [ A, B] \Vert .
\end{equation*}
Here, $[A,B]=A B - B A$ denotes the commutator. The constant $C_{\varphi}$ only depends on $\varphi$.
\end{lemma}

Applying this lemma, we immediately obtain

\begin{lemma}\sl \label{a26}
Let $f \in C^{0} ( \T )$ and $\varphi \in S^{\rho} (\R)$ with $\rho \leq 0$. Then
\begin{equation*}
[ \varphi (\varepsilon D ) , f ] = o_{\varepsilon \to 0} (1).
\end{equation*}
\end{lemma}

\begin{proof}
By Weierstrass's theorem, there exist $f_{k} \in C^{1} (\T )$ satisfying $f_{k} \to f$ in $L^{\infty} (\T )$. Then, viewed as operators, we have $f_{k} \to f$ in ${\mathcal L} (L^{2} (\T ))$. Remark that $[ \varepsilon D , f_{k}] = -\varepsilon i f_{k}'$. From Lemma \ref{a25}, we obtain
\begin{equation*}
\Vert [ \varphi (\varepsilon D ) , f_{k} ] \Vert \leq \varepsilon C_{\varphi} \Vert f_{k} ' \Vert_{\infty} .
\end{equation*}
Then, using that $\varphi$ is bounded,
\begin{align*}
[ \varphi (\varepsilon D ) , f ] =& [ \varphi (\varepsilon D ) , f_{k} ] + o_{k \to \infty} (1)  = \CO^{k} ( \varepsilon ) + o_{k \to \infty} (1)  \\
=& o_{\varepsilon \to 0} (1) ,
\end{align*}
since $[ \varphi (\varepsilon D ) , f ]$ does not depend on $k$.
\end{proof}

\Subsection{Essential spectrum of the product of Toeplitz operators}

Here, we recall the theorem of Douglas in our setting. In fact, the result of Douglas is true in a more general framework (see \cite[Theorem 4.5.10]{Ni02_01}). We shall give in section \ref{a41} an alternative proof of the following theorem, more related to our approach.

\begin{theorem}[{\bf Douglas}]\sl \label{a24}
Let $f,g \in C^{0} (\T , \R)$ with $g \geq 0$. The bounded self-adjoint operator $T (g)^{1/2} T (f) T (g)^{1/2}$ satisfies on $\im \Pi$
\begin{equation*}
\sigma_{{\rm ess}} \big( T (g)^{1/2} T (f) T (g)^{1/2} \big) = \big[ \inf ( f g) ,\sup ( f g) \big] .
\end{equation*}
Here, $\sigma_{{\rm ess}} (A)$ denotes the essential spectrum of $A$.
\end{theorem}

In Theorem \ref{a24}, the operator $T (g)^{1/2} T (f) T (g)^{1/2}$ is viewed as an operator on $\im \Pi$. On $L^{2} ( \T )$, this operator is a block diagonal operator with respect to the orthogonal sum $L^{2} = \im \Pi \oplus^{\perp} \im (1 - \Pi )$ and is equal to $0$ on $\im ( 1 - \Pi )$. In particular, we have

\begin{remark}\sl \label{a3}
If the operator $T (g)^{1/2} T (f) T (g)^{1/2}$ is viewed on $L^{2} (\T )$, we have
\begin{equation*}
\sigma_{{\rm ess}} \big( T (g)^{1/2} T (f) T (g)^{1/2} \big) = \big[ \inf ( f g) ,\sup ( f g) \big] \cup \{ 0 \} .
\end{equation*}
\end{remark}

\section{Proof of Theorem \ref{a8}}
\label{a51}

The goal of this section is to prove Theorem \ref{a8}. First of all, one can observe that part {\it ii)} clearly implies {\it iii)}. In the next subsection, we first show that {\it i)} implies {\it ii)}.

\subsection{The implication {\it i)} gives {\it ii)}}
\label{subkl}

\begin{lemma}\sl \label{lemmasp}
Let $f,g \in C^{0} (\T , \R)$ with $g \geq 0$. Then,
\begin{equation*}
T_{n} (g)^{1/2} T_{n} (f) T_{n} (g)^{1/2} \longrightarrow T (g)^{1/2} T (f) T (g)^{1/2} 
\end{equation*}
strongly on $L^{2} (\T )$. If $\lambda$ belongs to the spectrum of $T(f)T(g)$ on $\im \Pi$, then there exists an eigenvalue $\lambda_{n}$ of $T_{n}(f) T_{n}(g)$ on $\im \Pi_{n}$ such that $\lambda_{n} \to \lambda$.
\end{lemma}

\begin{proof}
For a sequence of bounded operators $(A_n)$, we shall denote
\begin{equation*}
A_{n} \overset{s}{\longrightarrow} A ,
\end{equation*}
if $A_{n}$ converges strongly to $A$. Since $\Pi_{n} \overset{s}{\longrightarrow} \Pi$, 
it follows from Lemma III.3.8 of \cite{Ka95_01} that for all $f\in L^{\infty}(\T)$,
$T_{n} (f) \overset{s}{\longrightarrow} T (f)$. 
In particular, from Problem VI.14 of \cite{ReSi80_01} (see also Theorem VI.9 of \cite{ReSi80_01}), 
$T_{n} (g)^{1/2} \overset{s}{\longrightarrow} T (g)^{1/2}$. 
Consequently, we deduce from Lemma III.3.8 of \cite{Ka95_01} that
\begin{equation} \label{a27}
T_{n} (g)^{1/2} T_{n} (f) T_{n} (g)^{1/2} \overset{s}{\longrightarrow} T (g)^{1/2} T (f) T (g)^{1/2} ,
\end{equation}
on $L^{2} (\T )$. In particular, we obtain on $\im \Pi$
\begin{equation*}
T_{n} (g)^{1/2} T_{n} (f) T_{n} (g)^{1/2} + M ( \Pi - \Pi_{n} ) \overset{s}{\longrightarrow} T (g)^{1/2} T (f) T (g)^{1/2} ,
\end{equation*}
for all $M \in \R$. We choose $\mu = \Vert f \Vert_{\infty} \Vert g \Vert_{\infty}$ and $M = \mu +1$. Therefore, it follows from Corollary VIII.1.6 of \cite{Ka95_01} that 
\begin{equation*}
T_{n} (g)^{1/2} T_{n} (f) T_{n} (g)^{1/2} + M ( \Pi - \Pi_{n} ) \longrightarrow T (g)^{1/2} T (f) T (g)^{1/2} 
\end{equation*}
strongly in the generalized sense on $\im \Pi$.

Consequently, Lemma \ref{a30} and Theorem VIII.1.14 of \cite{Ka95_01} imply that, for each $\lambda$ belonging to $\sigma(T(f)T(g))$ on $\im \Pi$, there exists an eigenvalue $\lambda_{n}$ of $T_{n} (g)^{1/2} T_{n} (f) T_{n} (g)^{1/2} + M ( \Pi - \Pi_{n} )$ on $\im \Pi$ such that $\lambda_{n} \to \lambda$. Since $\Vert T (g)^{1/2} T (f) T (g)^{1/2} \Vert \leq \mu$, we necessarily have $\lambda \in [ - \mu , \mu ]$ and then $M \geq \vert \lambda \vert +1$. In particular, for $n$ large enough, $M > \vert \lambda_{n} \vert + 1/2$. Therefore, $\lambda_{n}$ is an eigenvalue of $T_{n} (g)^{1/2} T_{n} (f) T_{n} (g)^{1/2}$ on $\im \Pi_{n}$ because
\begin{equation*}
T_{n} (g)^{1/2} T_{n} (f) T_{n} (g)^{1/2} + M ( \Pi - \Pi_{n} ) = T_{n} (g)^{1/2} T_{n} (f) T_{n} (g)^{1/2} \oplus^{\perp} M ( \Pi - \Pi_{n} ) ,
\end{equation*}
is a block diagonal operator with respect to the orthogonal sum $\im \Pi = \im \Pi_{n} \oplus^{\perp} \im (\Pi - \Pi_{n})$.
\end{proof}

\Subsection{The implication {\it iii)} gives {\it i)}}
\label{subloc}

Let $\lambda_{N}$ be a sequence of eigenvalues of $T_{N} (f) T_{N} (g)$ such that $\lambda_{N} \to \lambda \in \R$. Here $N$ is a subsequence of $\N$ and we have to show that $\lambda$ is in the spectrum of $T (f) T (g)$. From Lemma \ref{a30} and Theorem \ref{a24}, we know that $[ \inf ( f g) ,\sup ( f g) ]$ is always inside the spectrum of $T (f) T (g)$. Thus, we can assume that
\begin{equation} \label{a20}
\lambda \notin \big[ \inf ( f g) ,\sup ( f g) \big] .
\end{equation}

By Weierstrass's theorem, there exists a sequence of $C^{\infty} (\T )$ functions $(f_k)$ such that $f_{k} \to f$ in $L^{\infty} ( \T )$ and $\supp \widehat{f_{k}} \subset [ - k , k ]$. We also consider $(g_{k})$ a sequence corresponding to $g$ with the same properties mutatis mutandis. In particular, for all $n \in \N$,
\begin{equation}\label{a1}
T_{n} (f) = T_{n}(f_{k}) + o_{k \to \infty} (1) \qquad \text{ and } \qquad T (f) = T (f_{k}) + o_{k \to \infty} (1) .
\end{equation}
Recall that, by definition, a $o_{k \to \infty} (1)$ is uniform with respect to $n$.

Finally, let $u_{N} \in \im \Pi_{N}$ be an eigenvector of $T_{N} (f) T_{N} (g)$ associated with $\lambda_{N}$ and satisfying $\Vert u_{N} \Vert =1$. From \eqref{a1},
\begin{gather}
T_{N}(f)T_{N}(g) u_{N} = \lambda_{N} u_{N} = \lambda u_{N} + o_{N \to \infty}(1)    \label{a13}  \\
T_{N} (f_{k})  T_{N} (g_{k}) u_{N} = \lambda u_{N} + o_{N \to \infty} (1) + o_{k \to \infty} (1) . \label{a2}
\end{gather}
In the following, we denote $D_{n} = n^{-1} D$.

\Subsubsection{Localization of the eigenvectors}

\begin{lemma}\sl \label{lemmaloc}
Let $\varphi \in C^{\infty}_{0}(]0,1[)$. Then, in $L^{2} (\T )$ norm,
\begin{equation*}
\varphi(D_{N}) u_{N} = o_{N \to \infty} (1).
\end{equation*}
\end{lemma}

In \cite{BoGuCr07_01}, B\"{o}ttcher, Crespo and Guti\'errez-Guti\'errez have established that, for any $\varepsilon (n) >0$ with $n \varepsilon (n)$ going to infinity,
\begin{equation*}
\max_{n \varepsilon (n) < k < n (1 - \varepsilon (n)) } \Vert (T_n(fg)-T_n(f)T_n(g)) e_{k} \Vert \longrightarrow 0 ,
\end{equation*}
where $(e_{k})$ is the standard basis of $l^{2} ( \N )$. Here, we will show and use
\begin{equation*}
\Vert ( f g -T_n(f)T_n(g)) ( \Pi_{(1 - \varepsilon) n} - \Pi_{\varepsilon n}) \Vert \longrightarrow 0 ,
\end{equation*}
in operator norm for all $0< \varepsilon < 1/2$.

\begin{proof}
From Lemma \ref{a26}, we have
\begin{align}
\varphi (D_{N}) T_{N}(f) =& \varphi(D_{N}) {\bf 1}_{[0 ,1]} (D_{N}) f {\bf 1}_{[0 ,1]} (D_{N}) = \varphi (D_{N}) f {\bf 1}_{[0 ,1]} (D_{N})  \nonumber  \\
=& f \varphi (D_{N}) {\bf 1}_{[0 ,1]} (D_{N}) + o_{N \to \infty} (1)  \nonumber  \\
=& f \varphi (D_{N}) + o_{N \to \infty} (1).   \label{a16}
\end{align}
Applying two times this estimate, we obtain
\begin{align*}
\varphi (D_{N}) T_{N}(f)T_{N}(g)u_{N} 
&= f \varphi (D_{N}) T_{N} (g) u_{N} + o_{N \to \infty} (1)  \\
&= f g \varphi (D_{N}) u_{N} + o_{N \to \infty} (1).
\end{align*}
Then, \eqref{a13} gives
\begin{equation*}
( f g - \lambda ) \varphi (D_{N}) u_{N} = o_{N \to \infty} (1).
\end{equation*}
Since $\lambda \notin [ \inf ( f g) ,\sup ( f g) ]$, the function $( f g - \lambda )^{-1}$ is in $L^{\infty} (\T )$ and the lemma follows from the last equation.
\end{proof}

Now, we take $\varphi \in C^{\infty}_{0}(]0,1[ , [0 ,1])$ such that $\varphi =1$ near $[\varepsilon, 1 - \varepsilon ]$ for $\varepsilon >0$ small enough (we choose $\varepsilon = 1/8$). Let $\varphi^{-} \in C^{\infty}_{0} ([-\varepsilon, 2 \varepsilon ] , [0 ,1])$ and $\varphi^{+} \in C^{\infty}_{0} ([ 1-2 \varepsilon , 1+ \varepsilon ], [0 ,1] )$ be two functions such that
\begin{equation*}
\varphi^{-} + \varphi + \varphi^{+} = 1 ,
\end{equation*}
in the neighborhood of $[0,1]$. Set
\begin{equation*}
u_{N}^{\pm} = \varphi^{\pm} (D_{N}) u_{N} = \varphi^{\pm} (D_{N}) {\bf 1}_{[0 ,1]} (D_{N}) u_{N}.
\end{equation*}
As $\Vert u_{N} \Vert =1$, it follows from Lemma \ref{lemmaloc} that 
\begin{equation}
\label{ump}
\Vert u_{N}^{-} + u_{N}^{+} \Vert = 1 + o_{N \to \infty} (1).
\end{equation}
In particular, we can assume, up to the extraction of a subsequence, that
\begin{equation*}
\forall N \quad \Vert u_{N}^{-} \Vert \geq 1/3 \qquad \text{ or } \qquad \forall N \quad \Vert u_{N}^{+} \Vert \geq 1/3 .
\end{equation*}
In the next section, we will suppose
\begin{equation} \label{a19}
\Vert u_{N}^{-} \Vert \geq 1/3.
\end{equation}
The case $\Vert u_{N}^{+} \Vert \geq 1/3$ follows essentially the same lines and is treated in Section \ref{a17}. But before, we show that $u_{N}^{-}$ and $u_{N}^{+}$ are both quasimodes of $T_{N} (f) T_{N} (g)$ (this means that they are eigenvectors modulo a small term).

\begin{lemma}\sl \label{a15}
We have
\begin{equation*}
T_{N} (f_{k}) T_{N} (g_{k}) u_{N}^{\pm} = 
\lambda u_{N}^{\pm} + o_{k \to \infty}(1) + o_{N \to \infty}^{k}(1).
\end{equation*}
\end{lemma}

\begin{proof}
As in \eqref{a16}, using Lemma \ref{a26}, we get
\begin{align}
T_{N} (f_{k}) T_{N} (g_{k}) u_{N}^{\pm} =& {\bf 1}_{[0 ,1]} (D_{N}) f_{k} {\bf 1}_{[0 ,1]} (D_{N}) g_{k} {\bf 1}_{[0 ,1]} (D_{N}) \varphi^{\pm} (D_{N}) u_{N}  \nonumber  \\
=& {\bf 1}_{[0 ,1]} (D_{N}) f_{k} {\bf 1}_{[0 ,1]} (D_{N}) g_{k} \varphi^{\pm} (D_{N}) {\bf 1}_{[0 ,1]} (D_{N}) u_{N}   \nonumber  \\
=& {\bf 1}_{[0 ,1]} (D_{N}) f_{k} {\bf 1}_{[0 ,1]} (D_{N}) \varphi^{\pm} (D_{N}) g_{k} {\bf 1}_{[0 ,1]} (D_{N}) u_{N} + o_{N \to \infty}^{k} (1)  \nonumber  \\
=& {\bf 1}_{[0 ,1]} (D_{N}) \varphi^{\pm} (D_{N}) f_{k} {\bf 1}_{[0 ,1]} (D_{N}) g_{k} {\bf 1}_{[0 ,1]} (D_{N}) u_{N} + o_{N \to \infty}^{k} (1)  \nonumber  \\
=& \varphi^{\pm} (D_{N}) T_{N} (f_{k}) T_{N} (g_{k}) u_{N} + o_{N \to \infty}^{k} (1) .
\end{align}
The lemma follows from \eqref{a2} and the last identity.
\end{proof}

\Subsubsection{Concentration near the low frequencies}

Here, we assume \eqref{a19} and we prove that $u_{N}^{-}$, viewed as an element of $\im \Pi$, is a quasimode of $T (f_{k}) T (g_{k})$.

\begin{lemma}\sl \label{a5}
For $4 k \leq N$, we have
\begin{equation*}
T (f_{k}) T (g_{k}) u_{N}^{-} = T_{N} (f_{k}) T_{N} (g_{k}) u_{N}^{-} .
\end{equation*}
\end{lemma}

\begin{remark}\sl
In fact, for $4 k \leq N \leq n$, we have
\begin{equation*}
T_{n} (f_{k}) T_{n} (g_{k}) u_{N}^{-} = T_{N} (f_{k}) T_{N} (g_{k}) u_{N}^{-} .
\end{equation*}
\end{remark}

\begin{proof}
Recall that, if $u ,v$ are two functions of $L^{2} ( \T )$ such that $\supp \widehat{u} \subset [a,b]$ and $\supp \widehat{v} \subset [c,d]$, then $\supp \widehat{u v} \subset [a+c ,b+d]$. By definition,
\begin{equation} \label{a10}
T (f_{k}) T (g_{k}) u_{N}^{-} = \Pi f_{k} \Pi g_{k} \Pi \varphi^{-} (D_{N}) u_{N} = \Pi f_{k} \Pi g_{k} \Pi_{N} \varphi^{-} (D_{N}) u_{N} .
\end{equation}
Since $\supp \widehat{g_{k}} \subset [-k , k]$ and $\supp {\mathcal F} ( \Pi_{N}\varphi^{-} (D_{N}) u_{N} ) \subset [0 , N/4]$, the Fourier transform of the function $g_{k} \Pi_{N}\varphi^{-} (D_{N}) u_{N}$ is supported inside $[-k , N/4 + k] \subset [ -k , N ]$. In particular,
\begin{equation} \label{a11}
\Pi g_{k} \Pi_{N} \varphi^{-} (D_{N}) u_{N} = \Pi_{N} g_{k} \Pi_{N} \varphi^{-} (D_{N}) u_{N} ,
\end{equation}
and the Fourier transform of this function is supported inside $[0, N/4 + k]$. As before, the Fourier transform of 
\begin{equation*}
f_{k} \Pi_{N} g_{k} \Pi_{N} \varphi^{-} (D_{N}) u_{N}
\end{equation*}
is supported inside $[-k , N/4 + 2 k] \subset [ - k , N]$. Then
\begin{equation} \label{a12}
\Pi f_{k} \Pi_{N} g_{k} \Pi_{N} \varphi^{-} (D_{N}) u_{N} = \Pi_{N} f_{k} \Pi_{N} g_{k} \Pi_{N} \varphi^{-} (D_{N}) u_{N} .
\end{equation}
The lemma follows from \eqref{a10}, \eqref{a11} et \eqref{a12}.
\end{proof}

From \eqref{a1}, Lemma \ref{a15} and Lemma \ref{a5}, we get
\begin{equation} \label{a4}
T (f) T (g) u_{N}^{-} =  \lambda u_{N}^{-} + o_{k \to \infty} (1) + o_{N \to \infty}^{k} (1) ,
\end{equation}
for $4 k \leq N$. If $\lambda \notin \sigma ( T (f) T (g) )$, the operator $T (f) T (g) - \lambda$ is invertible and then
\begin{equation*}
u_{N}^{-} = o_{k \to \infty} (1) + o_{N \to \infty}^{k} (1) .
\end{equation*}
From \eqref{a19}, we obtain $1/3 \leq o_{k \to \infty} (1) + o_{N \to \infty}^{k} (1)$. Taking $k$ large enough and after $N$ large enough, it is clear that this is impossible. Thus,
\begin{equation*}
\lambda \in \sigma ( T (f) T (g) ) ,
\end{equation*}
which implies Theorem \ref{a8} under the assumption \eqref{a19}.

\Subsubsection{Concentration near the high frequencies}
\label{a17}

We replace the assumption \eqref{a19} by $\Vert u_{N}^{+} \Vert \geq 1/3$. Let
\begin{equation*}
\begin{aligned}
J : \\
{}^{}
\end{aligned}
\left\{ \begin{aligned}
&L^{2} (\T ) \ {\hbox to 14mm{\rightarrowfill}} &&L^{2} ( \T )  \\
&f (x) &&f(-x)
\end{aligned} \right.
\end{equation*}
Remark that $J ( u v ) = J (u) J (v)$. Using the notation $\Pi_{[a,b]} = {\bf 1}_{[a,b]} (D)$, we have $\Pi_{[a,b]} J = J \Pi_{[-b,-a]}$ and $\Pi_{[a,b]} e^{i c x} = e^{i c x} \Pi_{[a-c,b-c]}$. Combining these identities with Lemma \ref{a15}, we get
\begin{align}
T_{N} (J f_{k}) T_{N} (J g_{k} ) e^{i N x} ( J u_{N}^{+} ) =& \Pi_{[0,N]} (J f_{k}) \Pi_{[0,N]} (J g_{k}) \Pi_{[0,N]} e^{i N x} ( J u_{N}^{+} )   \nonumber \\
=& e^{i N x} \Pi_{[-N,0]} (J f_{k}) \Pi_{[-N,0]} (J g_{k}) \Pi_{[-N,0]} ( J u_{N}^{+} )   \nonumber \\
=& e^{i N x} \Pi_{[-N,0]} (J f_{k}) \Pi_{[-N,0]} (J g_{k}) J \Pi_{[0,N]} u_{N}^{+}   \nonumber \\
=& e^{i N x} \Pi_{[-N,0]} (J f_{k}) \Pi_{[-N,0]} J g_{k} \Pi_{[0,N]} u_{N}^{+}   \nonumber \\
=& e^{i N x} \Pi_{[-N,0]} (J f_{k}) J \Pi_{[0,N]} g_{k} \Pi_{[0,N]} u_{N}^{+}   \nonumber \\
=& e^{i N x} \Pi_{[-N,0]} J f_{k} \Pi_{[0,N]} g_{k} \Pi_{[0,N]} u_{N}^{+}   \nonumber \\
=& e^{i N x} J \Pi_{[0,N]} f_{k} \Pi_{[0,N]} g_{k} \Pi_{[0,N]} u_{N}^{+}   \nonumber \\
=& \lambda e^{i N x} (J u_{N}^{+}) + o_{k \to \infty} (1) + o_{N \to \infty}^{k} (1) .
\end{align}
In particular, $\widetilde{u}_{N}^{-} = e^{i N x} ( J u_{N}^{+} )$ satisfies $\Vert \widetilde{u}_{N}^{-} \Vert \geq 1/3$,
\begin{equation*}
T_{N} (J f_{k}) T_{N} (J g_{k} ) \widetilde{u}_{N}^{-} = \lambda \widetilde{u}_{N}^{-} + o_{k \to \infty} (1) + o_{N \to \infty}^{k} (1) ,
\end{equation*}
and the support of the Fourier transform of $\widetilde{u}_{N}^{-}$ is inside $[0 , N/4]$. In particular, we can apply the method developed in the case $\Vert u_{N}^{-} \Vert \geq 1/3$. The unique difference is that $f,g$ are replaced by $(J f) , (J g)$. Then we obtain
\begin{equation*}
\lambda \in \sigma \big( T (J f) T (J g) \big) .
\end{equation*}
Theorem \ref{a8} follows from the following lemma and $\sigma ( T (f) T (g) ) = \sigma ( T (g) T (f) )$ (The spectrum of $T (f) T (g)$ is real from Lemma \ref{a30} and $( T (f) T (g) -z )^{*} = T (g) T (f) - \overline{z}$.)

\begin{lemma}\sl
Let $f, g \in L^{\infty} ( \T )$. Then
\begin{equation*}
\sigma ( T (J f) T (J g) ) = \sigma ( T (g) T (f) ) .
\end{equation*}
\end{lemma}

\begin{proof}
For $A$ a bounded linear operator on $L^{2}$, we define $A^{t}$ by
\begin{equation*}
( A^{t} u , v ) = ( u , \overline{A \overline{v}} ) ,
\end{equation*}
for all $u , v \in L^{2}$. Simple calculi give $f^{t} = f$, $\Pi_{[a,b]}^{t} = \Pi_{[- b,- a]}$, $(A B)^{t} = B^{t} A^{t}$ and then
\begin{equation*}
T (f)^{t} = \big( \Pi_{[0 , + \infty[} f \Pi_{[0 , + \infty[} \big)^{t} = \Pi_{] - \infty , 0 ]} f \Pi_{] - \infty , 0 ]} .
\end{equation*}
By the same way, since $J= J^{*} = J^{-1}$,
\begin{equation*}
J \Pi_{] - \infty , 0 ]} f \Pi_{] - \infty , 0 ]} J = \Pi_{[0, + \infty [} (J f) \Pi_{[0, + \infty [} = T (J f) .
\end{equation*}
Combining these identities concerning ${}^{t}$ and $J$, we get
\begin{align}
J ( T (J f) T (J g) )^{t} J^{-1} =& J \big( \Pi_{] - \infty , 0 ]} (J g) \Pi_{] - \infty , 0 ]} \big) \big( \Pi_{] - \infty , 0 ]} (J f) \Pi_{] - \infty , 0 ]} \big) J \nonumber \\
=& T (g) T (f) .
\end{align}
Since $J A^{t} J -z = J ( A -z)^{t} J$, $A$ and $J A^{t} J$ have the same spectrum and the lemma follows.
\end{proof}

\section{Proof of the other results}
\label{a52}

\Subsection{Proof of Lemma \ref{a30}}

In fact, we will prove the following result: Let $A, B$ be two bounded operators on a Hilbert space. Assume that $B$ is self-adjoint with $B \geq 0$. Then
\begin{equation}\label{a28}
\sigma (A B ) = \sigma \big( B^{1/2} A B^{1/2} \big) .
\end{equation}

Recall that, in a ring, $1 - a b$ is invertible if and only if $1 - b a$ is invertible (the inverse of $1 - b a$ is then $1 + b ( 1 - a b)^{-1} a$). In particular, 
\begin{equation*}
\sigma (A B ) \setminus \{ 0 \} = \sigma \big( B^{1/2} A B^{1/2} \big) \setminus \{ 0 \} ,
\end{equation*}
and it remains to study the value $0$.

Assume first that $0 \in \sigma ( B^{1/2} )$. By Weyl's criterion (see \cite[Theorem VII.12]{ReSi80_01}), there exist $u_{n}$, with $\Vert u_{n} \Vert =1$, such that $B^{1/2} u_{n} \to 0$. Then, $A B u_{n} \to 0$ and $B^{1/2} A B^{1/2} u_{n}\to 0$. Thus, $0$ is in the spectrum of the two operators $A B$ and $B^{1/2} A B^{1/2}$.

Assume now that $0 \notin \sigma ( B^{1/2} )$. Then, $B^{1/2}$ is invertible. In particular, $A B$ is invertible if and only if $B^{1/2} A B^{1/2}$ is invertible. Thus, $0$ is in the spectrum of $A B$ if and only if $0$ is in the spectrum of $B^{1/2} A B^{1/2}$.

In the case of a finite dimensional Hilbert space, we use the formula
\begin{equation*}
\det ( A B  - \lambda ) = \det (B^{1/2} A B^{1/2} - \lambda ).
\end{equation*}
In particular, the operators $T_{n} (f) T_{n} (g)$ and $T_{n} (g)^{1/2} T_{n} (f) T_{n} (g)^{1/2}$ have the same eigenvalues with the same multiplicity.

\Subsection{Proof of Theorem \ref{a29}}

Here, we show that Theorem \ref{a8} implies Theorem \ref{a29}. We drop the subscript $(f,g)$ and denote $\lambda_{\limsup} = \limsup \lambda_{\max}^{n}$. By definition,
\begin{equation} \label{a33}
\lambda_{\max}^{n} \leq \lambda_{\limsup} + o_{n \to \infty} (1) ,
\end{equation}
and there exists a subsequence $N$ of $\N$ such that $\lambda_{\max}^{N} \to \lambda_{\limsup}$. In particular, using the implication {\it iii)}$\Rightarrow${\it i)} of Theorem \ref{a8}, $\lambda_{\limsup}$ is in the spectrum of $T (f) T (g)$ and then
\begin{equation} \label{a34}
\lambda_{\limsup} \leq \lambda_{\max} .
\end{equation}
On the other hand, according to the implication {\it i)}$\Rightarrow${\it ii)} of Theorem \ref{a8}, there exists $\lambda_{n} \in \sigma ( T_{n} (f) T_{n} (g) )$ such that $\lambda_{n} \to \lambda_{\max}$. Thus,
\begin{equation} \label{a35}
\lambda_{\max} + o_{n \to \infty} (1) \leq \lambda_{n} \leq \lambda_{\max}^{n}.
\end{equation}
Combining \eqref{a33}, \eqref{a34} and \eqref{a35}, we finally obtain $\lambda_{\max}^{n} \to \lambda_{\max} = \lambda_{\limsup}$.

\Subsection{Proof of Example \ref{a40}}
\label{a38}

The goal of this section is to study the extrema of the limiting spectrum of
the product
$T_n(f)T_n(g)$ in the particular case
\begin{equation*}
f(x)=a+\cos(x) \qquad \text{ and } \qquad g(x)=\frac{1}{1+\theta^2-2\theta
\cos(x)}
\end{equation*}
where $a, \theta \in \dR$ and $\vert \theta \vert<1$.

First of all, one can observe
that it is more convenient to work with the inverse
of $T_n(g)$. As a matter of fact, $T_n(g)^{-1}$ is a tridiagonal matrix quite
similar to $T_n(g^{-1})$ except that, at the two diagonal corners of $T_n(g^{-1})$, the coefficient
$1+\theta^2$ is replaced by $1$
\begin{equation*}
T_{n} (g)^{-1}  = \left( \begin{array}{cccc}
1 & - \theta & 0 & \ldots   \\
- \theta & 1 + \theta^{2} & - \theta & \ldots  \\
\ldots & \ldots & \ldots & \ldots   \\
\ldots & - \theta & 1+ \theta^{2} & - \theta \\
\ldots & 0 & - \theta & 1 \\
\end{array} \right) .
\end{equation*}

It is not hard to see that $\det(T_n(g)^{-1} ) = 1 - \theta^2$.
In order to find the eigenvalues $\lambda$ of the product $T_n(f)T_n(g)$, it is equivalent to calculate the zeros of its characteristic polynomial which correspond also to the zeros of $\det(M_{n} (t))$ where
\begin{equation*}
M_{n} (t) = t T_{n} (f) -  T_{n} (g)^{-1} ,
\end{equation*}
with $t = 1/ \lambda$. As $T_n(f)$ and $T_n(g)^{-1}$ are both tridiagonal matrices, we can easily compute 
$\det(M_n (t))$.

Via the same lines than in \cite[Lemma 11]{BeGaRo97_01}, we find that for $n$ large enough, 
$M_n (t)$ is negative definite only on the domain $\mathcal{D} =\mathcal{D}_1 \cup \mathcal{D}_2$
where
\begin{equation*}
\mathcal{D}_1 =\{-2\theta^2<p\leq -\theta^2 \ \text{and}\ q^2<-4\theta^2(p+\theta^2)\} \quad \text{ and } \quad \mathcal{D}_2 =\{p<- 2\theta^2 \ \text{and}\ p<-|q|\},
\end{equation*}
with $p=a t-(1+\theta^2)$ and $q= t +2\theta$. In term of the variable $\lambda$, the inverses of the boundaries of $\mathcal{D}$ give the extrema of the spectrum of the $T(f)T(g)$ {\it i.e.} $\lambda_{\max}(f,g)$ and $\lambda_{\min}(f,g)$. After some tedious but straightforward calculations, we obtain three inverses of the boundaries
\begin{equation*}
\frac{a-1}{(1+\theta)^2},
\hspace{1.5cm}
\frac{a+1}{(1-\theta)^2},
\hspace{1.5cm}
-\frac{1}{4\theta(1+a\theta)}.
\end{equation*}
Two of them coincide with $\inf(fg)$ and $\sup(fg)$.
It only depends on the location of $a$ with respect to
$-(1+\theta^2)/(2\theta)$. The last one can be $\lambda_{\max} (f ,g)>\sup(fg)$
or $\lambda_{\min} (f ,g)<\inf(fg)$. It only depends on the sign of $\theta$ as well
as on the location of $a$ with respect to the interval $[a_{\theta},b_{\theta}]$.

\section{An altenative proof of Douglas's theorem}
\label{a41}

Let $\psi \in C^{\infty} (\R)$ satisfying $\psi = 1$ near $[ 2 , + \infty [$ and $\psi = 0$ near $]- \infty , 1 ]$. For $\varepsilon >0$, we have on $\im \Pi$
\begin{align}
T (g)^{1/2} T (f) T (g)^{1/2} =& T (g)^{1/2} \psi ( \varepsilon D) T (f) \psi ( \varepsilon D)  T (g)^{1/2} + \widetilde{R}_{\varepsilon}   \nonumber  \\
=& T (g)^{1/2} \psi ( \varepsilon D) f \psi ( \varepsilon D)  T (g)^{1/2} + \widetilde{R}_{\varepsilon} , \label{a22}
\end{align}
where
\begin{equation*}
\widetilde{R}_{\varepsilon} = T (g)^{1/2} (1 - \psi ( \varepsilon D) ) T (f) \psi ( \varepsilon D)  T (g)^{1/2} + T (g)^{1/2} T (f) (1 - \psi ( \varepsilon D)) T (g)^{1/2} ,
\end{equation*}
is a self-adjoint operator of finite rank. Recall that if $A \geq 0$ is a bounded operator with $\Vert A \Vert \leq 1$, then
\begin{equation*}
A^{1/2} = \sum_{j=0}^{+ \infty} c_{j} (1 -A)^{j} ,
\end{equation*}
where $\Vert 1 - A \Vert \leq 1$ and $\sum_{j \geq 0} \vert c_{j} \vert \leq 2 < + \infty$ (see the proof of Theorem VI.9 of \cite{ReSi80_01}). On the other hand, Lemma \ref{a26} implies
\begin{align}
T (g) \psi (\varepsilon D ) =& \Pi g \Pi \psi (\varepsilon D ) = \Pi g \psi (\varepsilon D ) = \Pi \psi (\varepsilon D ) g + o_{\varepsilon \to 0} (1)  \nonumber  \\
=& \psi (\varepsilon D ) g + o_{\varepsilon \to 0} (1) ,
\end{align}
Then, for a fixed $\delta >0$ such that $\Vert T (g) \Vert \leq \Vert g \Vert_{\infty} < \delta^{-1}$, we have
\begin{align}
T (g)^{1/2} \psi ( \varepsilon D) =& \delta^{-1/2} T (\delta g)^{1/2} \psi ( \varepsilon D)     \nonumber  \\
=& \delta^{-1/2} \sum_{j=0}^{+ \infty} c_{j} (1 - T (\delta g))^{j} \psi ( \varepsilon D)   \nonumber  \\
=& \delta^{-1/2} \sum_{j=0}^{J} c_{j} (1 - T (\delta g))^{j} \psi ( \varepsilon D) + o_{J \to \infty} (1)    \nonumber \\
=& \delta^{-1/2} \psi ( \varepsilon D) \sum_{j=0}^{J} c_{j} (1 - \delta g)^{j} + o_{J \to \infty} (1) + o_{\varepsilon \to 0}^{J} (1)    \nonumber \\
=& \delta^{-1/2} \psi ( \varepsilon D) (\delta g)^{1/2} + o_{J \to \infty} (1) + o_{\varepsilon \to 0}^{J} (1)   \nonumber  \\
=& \psi ( \varepsilon D) g^{1/2} + o_{\varepsilon \to 0} (1) ,
\end{align}
since these quantities do not depend on $J$. Using this identity and its adjoint, \eqref{a22} becomes
\begin{align}
T (g)^{1/2} T (f) T (g)^{1/2} =& \psi ( \varepsilon D) f g \psi ( \varepsilon D) + \widetilde{R}_{\varepsilon} + o_{\varepsilon \to 0} (1)  \nonumber \\
=& T ( f g ) + R_{\varepsilon} + e_{\varepsilon} , \label{a23}
\end{align}
where $e_{\varepsilon} = o_{\varepsilon \to 0} (1)$ and
\begin{equation*}
R_{\varepsilon} = \widetilde{R}_{\varepsilon} + (\psi ( \varepsilon D) -1) T (f g) \psi ( \varepsilon D) + T (f g) ( \psi ( \varepsilon D) -1),
\end{equation*}
is a self-adjoint operator of finite rank. In particular, $e_{\varepsilon}$ is a self-adjoint operator. Since, on the image on $\Pi$,
\begin{equation*}
\min (f g ) \leq T (f g) \leq \max (f g),
\end{equation*}
we get $\sigma ( T (f g) + e_{\varepsilon} ) \subset [ \min (f g ) - o_{\varepsilon \to 0} (1) , \max (f g) + o_{\varepsilon \to 0} (1)]$. Since $R_{\varepsilon}$ is of finite rank, we obtain, from Weyl's theorem \cite[Theorem S.13]{ReSi80_01},
\begin{equation*}
\sigma_{{\rm ess}} \big( T (g)^{1/2} T (f) T (g)^{1/2} \big) = \sigma_{{\rm ess}} ( T (f g) + e_{\varepsilon} ) \subset [ \min (f g ) - o_{\varepsilon \to 0} (1) , \max (f g) + o_{\varepsilon \to 0} (1)] .
\end{equation*}
As the essential spectrum of $T (g)^{1/2} T (f) T (g)^{1/2}$ does not depend on $\varepsilon$, we get
\begin{equation} \label{a37}
\sigma_{{\rm ess}} \big( T (g)^{1/2} T (f) T (g)^{1/2} \big) \subset [ \min (f g ) , \max (f g) ] ,
\end{equation}
which is the first inclusion of the theorem.

Now, let $\varphi \in C^{\infty} ( [ -1 , 1] , [ 0 ,1])$ with $\Vert \varphi \Vert_{L^{2}} =1$. For $x_{0} \in \T$ and $\alpha , \beta \in \N$, we set
\begin{equation*}
u = \alpha^{1/2} \varphi \big( \alpha (x - x_{0} ) \big) e^{i \beta x}  \quad \text{ and } \quad v = \Pi u \in \im \Pi ,
\end{equation*}
which satisfies $\Vert u \Vert =1$. We have
\begin{align}
( 1 - \Pi ) u =& \alpha^{1/2} {\bf 1}_{] - \infty , 0 ]} (D) e^{i \beta x} \varphi \big( \alpha (x - x_{0} ) \big)   \nonumber  \\
=& \alpha^{1/2} e^{i \beta x} {\bf 1}_{] - \infty , - \beta ]} (D) \varphi \big( \alpha (x - x_{0} ) \big)  \nonumber  \\
=& \alpha^{1/2} e^{i \beta x} {\bf 1}_{] - \infty , - \beta ]} (D) (D + i)^{-M} ( D+i)^{M} \varphi \big( \alpha (x - x_{0} ) \big)   \nonumber  \\
=& \CO \big( \beta^{-M} \alpha^{M} \big), \label{a39}
\end{align}
in $L^{2}$ norm for any $M \in \N$. Moreover, for a continuous function $\ell$, we have
\begin{equation}
\ell u = \ell ( x_{0} ) \alpha^{1/2} \varphi \big( \alpha (x - x_{0} ) \big) e^{i \beta x} + o_{\alpha \to \infty} (1) ,
\end{equation}
in $L^{2}$ norm. Using that $\Vert T ( \ell )^{1/2} \Vert \leq \Vert \ell \Vert_{\infty}^{1/2}$, for all function $\ell \in L^{\infty}$ with $\ell \geq 0$, we get
\begin{align}
T (f) T (g) v =& \Pi f \Pi g \Pi u = \Pi f \Pi g u + \CO \big( \alpha \beta^{-1} \big)    \nonumber  \\
=& g (x_{0}) \Pi f \Pi u + \CO \big( \alpha \beta^{-1} \big) + o_{\alpha \to \infty} (1)   \nonumber  \\
=& g (x_{0}) \Pi f u + \CO \big( \alpha \beta^{-1} \big) + o_{\alpha \to \infty} (1)    \nonumber  \\
=& (f g) (x_{0}) \Pi u + \CO \big( \alpha \beta^{-1} \big) + o_{\alpha \to \infty} (1)   \nonumber  \\
=& (f g) (x_{0}) v + \CO \big( \alpha \beta^{-1} \big) + o_{\alpha \to \infty} (1) . \label{a18}
\end{align}

Taking $\beta = \alpha^{2} \to + \infty$, \eqref{a39} implies $\Vert v \Vert = 1 + o_{\alpha \to \infty} (1)$. On the other hand, \eqref{a18} gives
\begin{equation*}
T (f) T (g) v = (f g) (x_{0}) v + o_{\alpha \to \infty} (1) .
\end{equation*}
Then, $(f g) (x_{0}) \in \sigma ( T (f) T (g) ) = \sigma ( T (g)^{1/2} T (f) T (g)^{1/2} )$ from Lemma \ref{a30}. Therefore,
\begin{equation} \label{a36}
\big[ \inf (f g) , \sup (f g) \big] \subset \sigma \big( T (g)^{1/2} T (f) T (g)^{1/2} \big) .
\end{equation}

Recall that the essential spectrum of a self-adjoint bounded operator on an infinite Hilbert space is never empty. Therefore, if $\inf (f g) = \sup (f g)$, \eqref{a37} implies the theorem.

Assume now that $\inf (f g) < \sup (f g)$. Then $[ \inf (f g) , \sup (f g) ]$ is an interval with non empty interior. From the definition of the essential spectrum, this interval is necessarily inside the essential spectrum of $T (g)^{1/2} T (f) T (g)^{1/2}$. This achieves the proof of the second inclusion of the theorem.

\providecommand{\bysame}{\leavevmode\hbox to3em{\hrulefill}\thinspace}
\providecommand{\MR}{\relax\ifhmode\unskip\space\fi MR }
\providecommand{\MRhref}[2]{%
  \href{http://www.ams.org/mathscinet-getitem?mr=#1}{#2}
}
\providecommand{\href}[2]{#2}


\end{document}